\documentclass[11pt,oneside,english,a4]{amsart}
\usepackage[latin9]{inputenc}
\usepackage{amstext}
\usepackage{amsthm}
\usepackage{amssymb}

\makeatletter
\numberwithin{equation}{section}
\numberwithin{figure}{section}
  \theoremstyle{definition}
  \newtheorem{defn}{\protect\definitionname}
  \theoremstyle{remark}
  \newtheorem{notation}{\protect\notationname}
  \theoremstyle{plain}
  \newtheorem{fact}{\protect\factname}
  \theoremstyle{plain}
  \newtheorem{prop}{\protect\propositionname}
\theoremstyle{plain}
\newtheorem{thm}{\protect\theoremname}
  \theoremstyle{remark}
  \newtheorem{rem}{\protect\remarkname}
  \theoremstyle{plain}
  \newtheorem{lem}{\protect\lemmaname}
  \theoremstyle{plain}
  \newtheorem{cor}{\protect\corollaryname}
  \theoremstyle{plain}
  \newtheorem{question}{\protect\questionname}
  \newtheorem*{thm*}{\protect\theoremname}
   \newtheorem*{cor*}{\protect\corollaryname}

\def\Ind#1#2{#1\setbox0=\hbox{$#1x$}\kern\wd0\hbox to 0pt{\hss$#1\mid$\hss}
\lower.9\ht0\hbox to 0pt{\hss$#1\smile$\hss}\kern\wd0}

\def\Notind#1#2{#1\setbox0=\hbox{$#1x$}\kern\wd0\hbox to 0pt{\mathchardef
\nn="3236\hss$#1\nn$\kern1.4\wd0\hss}\hbox to
0pt{\hss$#1\mid$\hss}\lower.9\ht0
\hbox to 0pt{\hss$#1\smile$\hss}\kern\wd0}

\makeatother

\usepackage{babel}
  \providecommand{\definitionname}{Definition}
  \providecommand{\factname}{Fact}
  \providecommand{\lemmaname}{Lemma}
  \providecommand{\notationname}{Notation}
  \providecommand{\propositionname}{Proposition}
  \providecommand{\questionname}{Question}
  \providecommand{\remarkname}{Remark}
\providecommand{\corollaryname}{Corollary}
\providecommand{\theoremname}{Theorem}

\begin{document}
\global\long\def\Rr{\mathbb{R}}
 \global\long\def\dim#1{\mathrm{d}\left(#1\right)}
 \global\long\def\cl#1{\mathrm{cl}\left(#1\right)}
 \global\long\def\gcl#1{\mathrm{gcl}\left(#1\right)}
 \global\long\def\Aut{\mathrm{Aut}}
 \global\long\def\Autf{\mathrm{Autf}}
 \global\long\def\acleq{\mbox{\ensuremath{\mathrm{acl}}\ensuremath{\ensuremath{^{eq}}}}}
 \global\long\def\bdd{\mathrm{Bdd}}
 \global\long\def\FK#1#2{\left\langle {#1\atop #2}\right\rangle }
 \global\long\def\PK#1#2{\left({#1\atop #2}\right)}
\global\long\def\Age{\mathrm{Age}}
\global\long\def\cld#1#2{\mathrm{cl_{#2}}\left(#1\right)}
\global\long\def\clde#1#2#3{\mathrm{cl_{#2}}^{#3}\left(#1\right)}
\global\long\def\subf{\subseteq_{\mbox{fin}}}
 \global\long\def\ntp{\mbox{NTP}_{2}}
\global\long\def\nsop#1{\mbox{NSOP}_{#1}}
\global\long\def\sop{\mbox{SOP}}
\global\long\def\ac{\mbox{AC}}

\title{Convex Ramsey Matrices and Non-Amenability of Automorphism groups of generic structures } 
\authors{
\author{Omid Etesami}
\address{ School of Mathematics\\ Institute for Research in Fundamental Sciences (IPM)\\ P.O. Box 19395-5746\\ Tehran, Iran.}
\email{etesami@gmail.com }
\author{Zaniar Ghadernezhad}
\address{ School of Mathematics\\ Institute for Research in Fundamental Sciences (IPM)\\ P.O. Box 19395-5746\\ Tehran, Iran\\and \\Mathematisches Institut\\ Abteilung für Mathematische Logik\\ Eckerstr 1, Raum 304, D-79104\\ Freiburg im Breisgau}
\email{zaniar.ghadernezhad@math.uni-freiburg.de}
}
\thanks{ The second author would like to thank Katrin Tent and Matrin Ziegler for helpful comments and discussions on different parts of this article.}
\maketitle

\begin{abstract}
In this paper we prove that the automorphism groups of certain countable
generic structures are not amenable. We first prove the existence
of particular matrices that do not satisfy the convex Ramsey condition.
Then for a pair of elements in a smooth class, we introduce the property of forming a ``free-pseudoplane" in the generic structure. The existence
of such particular matrices and the correspondence in \cite{GKP-Amen} allow
us to prove the non-amenability of the automorphism group of a generic structure obtained from a smooth class with a pair that forms a free-pseudoplane. As an application we show that the automorphism group
of an ab-initio generic structure that is constructed using a pre-dimension function with irrational coefficients is not amenable.
\end{abstract}

\maketitle

\section{Introduction}

The study of amenable groups originated in the work of von Neumann in his analysis of the Banach-Tarski paradox. Since then, amenability, non-amenability, and paradoxicality have been studied for various groups appearing in different parts of mathematics. The definition of an amenable group can be considered for any topological group, although the original definition was first phrased only for locally compact Hausdorff groups. Let $G$ be a topological group. A $G$-flow is a continuous action of $G$ on a compact Hausdorff space. A group $G$ is amenable if every $G$-flow admits an invariant Borel probability measure. Well-known examples of amenable groups are finite groups, solvable groups, and locally compact abelian groups. 

The study of amenability of topological groups benefit from various viewpoints that range from analytic to combinatorial. The groups that we are considering in this paper are the automorphism groups of countable structures. The automorphism group of a countable first-order structure equipped with the point-wise convergence topology is a Polish group, which is also a closed subgroup of the symmetric group of its underlying set. Kechris, Pestov, and Todorcevic in \cite{KPT} established a very general correspondence which equates a stronger form of amenability, called extreme amenability, of the automorphism group of an ordered Fra\"iss\'e-limit structure with the Ramsey property of the class of its finite substructures. This is particularly interesting because in a suitable language any closed subgroup of the symmetric group of a countable set can be seen as the automorphism group of a  Fra\"iss\'e-limit structure.
 
 In the same spirit Moore in  \cite{TatchMoore2011} showed a correspondence between the automorphism groups of Fra\"iss\'e-limit structures and an another structural Ramsey property, called convex Ramsey property, which englobes F\o lner's existing treatment in the analytic approach. In this paper, we consider automorphism groups of  generic structures. A generic structure, similar to a Fra\"iss\'e-limit structure,  is constructed from a countable class of finite structures, called a smooth class, with an amalgamation property.  In \cite{GKP-Amen} following ideas
 of \cite{KPT,TatchMoore2011} it is shown that the amenability and
 extreme amenability of the automorphism group of a generic
 structure corresponds to the structural Ramsey type properties of its smooth class. 
 
It is shown in Theorem 32 in \cite{GKP-Amen} that the
automorphism group of a generic structure is amenable if and only
if the automorphism group has the convex Ramsey property with
respect to the smooth class (see Definition $\ref{def:convex ramsey}$).
Then the correspondence is used to show that if the smooth class of a generic
structure contains a certain pair of finite closed substructures, called a tree-pair (Definition 39 in
\cite{GKP-Amen}), then the automorphism group of the generic structure
is not amenable (see Theorem 40 in \cite{GKP-Amen}).  Existence of a tree-pair implies that there is an open subgroup of the automorphism group of the  generic structure that acts on a tree which is a substructure of the generic structure. 

The class of generic structures that are mainly considered are those that are obtained
from pre-dimension functions $\delta_{\alpha}$ where $\alpha\in\left(0,1\right)$
(see Subsection $\ref{subsec: pre-dim}$). They are originated
in the seminal work of Hrushovski in \cite{Hrunew} where he constructs a strongly minimal structure with a non-locally modular geometry that does not interpret an infinite group.
In Theorem 40 in \cite{GKP-Amen} it is shown that the generic structures that are
obtained from smooth classes of finite structures with pre-dimension
functions $\delta_{\alpha}$ where $\alpha$ is a rational number,
contain tree-pairs and hence the automorphism groups of their
generic structures are not amenable. However, for the generic structures
that are obtained from pre-dimension functions with irrational $\alpha$'s the statement of  Theorem 40 in \cite{GKP-Amen} do not hold. These generic structures
are of a particular interest since by \cite{MR1407480} their theory
is the zero-one law theory of graphs with the edge probability $n^{-\alpha}$
(see \cite{MR924703}).

In this paper, by exhibiting a combinatorial (or geometrical) criterion for a pair of elements in a smooth class, we show that the automorphism group of certain generic structures are not amenable. We first prove in Section \ref{sec:convex}, using probabilistic methods, the following theorem that guarantees the existence of certain matrices. 

\begin{thm*}
	For every $k \ge 1$, 
 there is an $n > k$ and 
 $n\times n$-matrix $\mathrm{X}$ whose entries are $0$ or $1$ such that $\mathrm{X}$ satisfies the $k$-configuration exhibiting
 condition but does not satisfy the convex Ramsey property.
\end{thm*}
We then  generalize Theorem 40 in \cite{GKP-Amen} and
introduce another sufficient condition for a pair of elements in a smooth class, called forming a \emph{free-pseudoplane}
(see Definition $\ref{def:imp}$) in the generic structure. The existence of such a pair implies that
the automorphism group of the generic structure is not amenable. More precisely, in Section \ref{sec:main} we prove the following theorem.
\begin{thm*}
	Suppose  $\mathcal{M}$ is the $\left(\mathcal{C},\leqslant\right)$-generic
	structure of
	a smooth class $\left(\mathcal{C},\leqslant\right)$ with AP (amalgamation property). Suppose for $n\in\mathbb{N}$ there is $A,B\in\mathcal{C}$
	such that $\left|\PK BA\right|=n$ and $\left(A;B\right)$ is a free
	$2$-pseudoplane and moreover, assume there is an $n\times m$-matrix
	$\mathrm{X}$ such that $\mathrm{X}$ satisfies the 2-configuration exhibiting
	condition but does not satisfy the convex Ramsey property. Then $\Aut\left(\mathcal{M}\right)$
	is not amenable.
\end{thm*}

Then finally using the above theorem, in Subsection \ref{subsec: pre-dim}  we prove the following
\footnote{ David M. Evans in an email correspondence in 2015 informed the second author that, using a different method, the automorphism groups of generic structures that are obtained from pre-dimension functions	with irrational coefficients and the $\omega$-categorical generic structures are not amenable.}:

\begin{cor*}
 The automorphism group
 of generic structures that are obtained from pre-dimension functions $\delta_\alpha$  where $\alpha$ is irrational are not amenable.
\end{cor*}

\subsection{Setting}

As we have mentioned in the introduction, there is a natural topology on the automorphism group of a countable
$\mathfrak{L}$-structure $M$; that is the pointwise convergence
topology. This topology turns $\Aut\left(M\right)$ into a topological
group and indeed a Polish group (see \cite{MR1066691}). Our focus in this paper is the automorphism groups of generic structures, a specific kind of
countable structures that are constructed from a class a finite structures with an amalgamation property. 

\subsubsection{Generic structures}
\begin{defn}
\label{def:smooth} Let $\mathcal{\mathfrak{L}}$ be a countable relational
language and $\mathcal{C}$ be a class of $\mathfrak{L}$-structures
which is closed under isomorphism and substructure. Assume $\emptyset\in \mathcal C$. Let $\leqslant$
be a reflexive and transitive relation on elements $A,B\in\mathcal{C}$
where $A\subseteq B$ and moreover, invariant under $\mathfrak{L}$-embeddings
that has the following property
\begin{itemize}
\item If $A,A_{1},A_{2}\in\mathcal{C}$ and $A_{1},A_{2}\subseteq A$, then
$A_{1}\leqslant A$ implies $A_{1}\cap A_{2}\leqslant A_{2}$.
\end{itemize}
The class $\mathcal{C}$ together with the relation $\leqslant$ is
called a \emph{smooth class}. For $A,B\in\mathcal{\mathcal{C}}$ if
$A\leqslant B$, then we say $A$ is\emph{ $\leqslant$-closed substructure}
of $B$, or simply $A$ is a \emph{ $\leqslant$-closed }in $B$. Moreover,
if $\mathcal{N}$ is an infinite $\mathfrak{L}$-structure such that
$A\subseteq N$, we write $A\leqslant\mathcal{N}$ whenever $A\leqslant B$
for every finite substructure $B$ of $N$ that contains $A$. We
say an embedding $\Gamma$ of $A$ into $\mathcal{N}$ is \emph{$\leqslant$-embedding}
if $\Gamma\left[A\right]\leqslant\mathcal{N}$. 
\end{defn}
\begin{notation}
Suppose $A,B,C$ are $\mathfrak{L}$-structures with $A,B\subseteq C$.
We write $AB$ for the $\mathfrak{L}$-substructure of $C$ with domain
$A\cup B$. For an $\mathfrak{L}$-structure $\mathcal{N}$, write
$\mbox{Age}\left(\mathcal{N}\right)$ for the set of all finite substructures
of $\mathcal{N}$, up to isomorphism. Write $\overline{\mathcal{C}}$
for the class of all $\mathfrak{L}$-structures $\mathcal{N}$ such
that $\Age\left(\mathcal{N}\right)\subseteq\mathcal{\mathcal{C}}$.
Suppose $A\in\mathcal{C}$ and let $N\in\bar{\mathcal{C}}$. Write
$\PK NA$ for the set of all $\leqslant$-embeddings of $A$ into
$N$ and write $\FK NA$ for the set of all finitely supported probability
measures on $\PK NA$. Suppose $X\leqslant Y\leqslant Z$ and let
$\mbox{\ensuremath{\mathsf{p}}}\in\FK ZY$. By $\FK{\mathsf{p}}X$
we mean the set 
\[
\left\{ \mathsf{q}\in\FK ZX:\exists\mathsf{r}\in\FK YX\forall\Gamma\in\PK YX\forall\Lambda\in\PK ZY,\mathsf{q}\left(\Lambda\circ\Gamma\right)=\mathsf{p}\left(\Lambda\right)\cdot\mathsf{r}\left(\Gamma\right)\right\} .
\]
\end{notation}
\begin{defn}
Let $\left(\mathcal{C},\leqslant\right)$ be a smooth class.
\begin{enumerate}
\item We say $\left(\mathcal{C},\leqslant\right)$ has the \emph{joint-embedding
property} (JEP) if for every two elements  $A,B\in\mathcal{C}$ there is $C\in\mathcal{C}$
such that $A,B\leqslant C$.
\item Suppose $A,B$ and $C$ are elements of $\mathcal{C}$ such that $A\leqslant B,C$.
The \emph{free-amalgam} of $B$ and $C$ over $A$ is a structure
consisting of the disjoint union of
$B$ and $C$ over $A$ whose only $\mathfrak{L}$-relations are those from
$B$ and $C$, denoted by $B\otimes_{A}C$.
\item We say $\left(\mathcal{C},\leqslant\right)$ has the \emph{$\leqslant$-amalgamation
property} (AP) if for every $A,B$ and $C$ elements of $\mathcal{C}$
with $A\leqslant B,C$, there is $D\in\mathcal{C}$ such that $B\leqslant D$
and $C\leqslant D$.
\item We say $\left(\mathcal{C},\leqslant\right)$ has the \emph{free-amalgamation
property} if for every $A,B$ and $C$ elements of $\mathcal{C}$
with $A\leqslant B,C$, then $B\otimes_{A}C\in\mathcal{C}$ and $B,C\leqslant B\otimes_{A}C$.
\end{enumerate}
\end{defn}
\begin{fact}
\label{f5} (See \cite{Wag1}) Let $\left(\mathcal{C},\leqslant\right)$
be a smooth class and suppose $A\in\mathcal{C}$ and $\mathcal{N}\in\overline{\mathcal{C}}$
are $\mathfrak{L}$-structures such that $A\subseteq\mathcal{N}$.
Then, there is a unique smallest $\leqslant$-closed set that contains
$A$ in $\mathcal{N}$. It is called $\leqslant$-closure of $A$
in $\mathcal{N}$, denoted by $\cld A{\mathcal{N}}$. 
\end{fact}
\begin{prop}
\label{prop:sgen} Suppose $\left(\mathcal{C},\leqslant\right)$ is
a smooth class with AP. Then there is a unique countable structure
$\mathcal{M}$, up to isomorphism, satisfying: 
\begin{enumerate}
\item $\Age\left(\mathcal{M}\right)=\mathcal{C}$; 
\item $\mathcal{M}=\bigcup_{i\in\omega}A_{i}$ where $\left(A_{i}:i\in\omega\right)$
is a chain of $\leqslant$-closed finite sets;
\item If $A\leqslant\mathcal{M}$ and $A\leqslant B\in\mathcal{C}$, then
there is an embedding $\Lambda:B\longrightarrow\mathcal{M}$ with
$\Lambda\upharpoonright_{A}=\mbox{id}_{A}$ and $\Lambda\left[B\right]\leqslant\mathcal{M}$. 
\end{enumerate}
\end{prop}
\begin{proof}
See \cite{Wag1}. 
\end{proof}
\begin{defn}
The structure $\mathcal{M}$, that is obtained in the above proposition,
is called the \emph{Fra\"iss\'e-Hrushovski }$\left(\mathcal{C},\leqslant\right)$\emph{-generic
}structure or simply the $\left(\mathcal{C},\leqslant\right)$\emph{-generic
}structure\emph{.}
\end{defn}
In \cite{GKP-Amen} it is
shown that the \emph{amenability} and \emph{extreme amenability} of
the automorphism group of a generic structure corresponds to structural Ramsey properties of its smooth class.
\begin{defn}
\label{def:convex ramsey} Suppose $\mathcal{M}$ is the $\left(\mathcal{C},\leqslant\right)$-generic
structure of a smooth class $\left(\mathcal{C},\leqslant\right)$.
We say $\Aut\left(\mathcal{M}\right)$ has the \emph{convex $\leqslant$-Ramsey
property with respect to} $\left(\mathcal{C},\leqslant\right)$ if
for every $A,B\in\mathcal{C}$ with $A\leqslant B$ and every $2$-coloring
function $f:\PK{\mathcal{M}}A\rightarrow\left\{ 0,1\right\} $, there
is $\mathsf{p}\in\FK{\mathcal{M}}B$ such that $\left|f\left(\mathsf{p}_{1}\right)-f\left(\mathsf{p}_{2}\right)\right|\leq\frac{1}{2}$
for every $\mathsf{q}_{1},\mathsf{q}_{2}\in\FK{\mathsf{p}}A$.
\end{defn}
The following correspondence has been proved in \cite{GKP-Amen}.
\begin{thm}
Suppose $\mathcal{M}$ is the $\left(\mathcal{C},\leqslant\right)$-generic
structure of a smooth class $\left(\mathcal{C},\leqslant\right)$.
Then $\Aut\left(\mathcal{M}\right)$ has the convex $\leqslant$-Ramsey
property\emph{ }with respect to $\left(\mathcal{C},\leqslant\right)$
if and only if $\Aut\left(\mathcal{M}\right)$ is amenable.
\end{thm}
Later in \cite{GKP-Amen} it has been shown that the convex $\leqslant$-Ramsey
property\emph{ }with respect to $\left(\mathcal{C},\leqslant\right)$
for $\Aut\left(\mathcal{M}\right)$ can be reformulated in the following manner: Namely the coloring matrices of all $2$-coloring functions of $\mathcal{M}$
have the convex Ramsey property (see Section 3 in \cite{GKP-Amen}) and
here we work with this latter formulation.

\subsubsection{Convex Ramsey matrices }

Before stating the definition of the convex Ramsey property for a
matrix, we need the following preliminary definitions.
\begin{defn}
\begin{enumerate}
\item An $\left(m\times1\right)$-matrix $\mathrm{W}$ is called a\emph{
Dirac-weight} \emph{matrix} if there is exactly one entry in $\mathrm{W}$
with value $1$ and exactly one entry with value $-1$ while all the
other entries of $\mathrm{W}$ are $0$.
\item A $1\times t$-matrix $\mathrm{P}=\begin{pmatrix}p_{1} &  & \cdots &  & p_{t}\end{pmatrix}_{1\times t}$
is called a \emph{probability matrix} if $p_{i}\geq0$ for every $1\leq i\leq t$,
and $\sum_{1\leq i\leq t}p_{i}=1$.
\end{enumerate}
\end{defn}
Here is the definition of a convex Ramsey matrix.
\begin{defn}
(Definition 27. in \cite{GKP-Amen}) Let $\mathrm{Y}=\begin{bmatrix}\mathrm{\bar{y}}_{1}\\
\\
\vdots\\
\\
\mathrm{\bar{y}}_{n}
\end{bmatrix}$ be an $\left(n\times m\right)$-matrix whose entries are $0$ or
$1$. Moreover, assume that no two rows of $\mathrm{Y}$ are the same. 
We say $\mathrm{Y}$ satisfies\emph{ }the\emph{ convex Ramsey condition}
if there is a $1\times t$-probability matrix $\mathrm{P}$ such that
for every $m\times1$ Dirac-weight matrix $\mathrm{W}$ we have
\[
\mathrm{P}\times\mathrm{Y}\times\mathrm{W}\leq\frac{1}{2}.
\]
\end{defn}
It is easy to give examples of matrices that satisfy the convex Ramsey
condition; for example a matrix with a constant row of $0$ or $1$
satisfies the convex Ramsey condition. A matrix with one column of
constant $0$ and one column of constant $1$ does not satisfy the
convex Ramsey condition. In order to show that the automorphism group
of a generic structure is amenable we need to show that every 2-coloring
matrix satisfies the convex Ramsey property. Conversely, if we show
there is a coloring matrix of the generic structure which does not
satisfy the convex Ramsey condition then the automorphism group is
not amenable.

\section{Convex Ramsey and $k$-configuration exhibiting matrices }
\label{sec:convex}
For a matrix to be a coloring matrix of a generic structure
some extra conditions are needed to be satisfied. The following condition
is inspired by Lemma 28 in \cite{GKP-Amen}.
\begin{defn}
	\label{def-conf}
Suppose $k\geq1$ is an integer. An $n\times m$-matrix $\mathrm{C}=\begin{bmatrix}\mathrm{\bar{c}}_{1}\\
\\
\vdots\\
\\
\mathrm{\bar{c}}_{n}
\end{bmatrix}$ whose entries are $0$ or $1$ and $m\geq k$. Then we say $\mathrm{C}$
satisfies\emph{ }the\emph{ $k$-configuration exhibition condition}
if for every $\sigma:\left\{ 1,\ldots,k\right\} \rightarrow\left\{ 0,1\right\} $
and every $l_{1}\cdots,l_{k}$ where $1\leq l_{1}<\cdots<l_{k}\leq m$,
there is $1\leq j\leq n$ such that $c_{jl_{i}}=\sigma\left(i\right)$
for $1\leq i\leq k$. 
\end{defn}
 With this new terminology, 
 that the following matrix satisfies the $1$-configuration exhibition
 condition but does not satisfy the convex Ramsey condition
 \[
 \begin{pmatrix}1 & 1 & 1 & 1 & 0 & 0\\
 1 & 1 & 1 & 0 & 1 & 0\\
 1 & 1 & 1 & 0 & 0 & 1\\
 0 & 1 & 1 & 0 & 0 & 0\\
 1 & 0 & 1 & 0 & 0 & 0\\
 1 & 1 & 0 & 0 & 0 & 0
 \end{pmatrix}_{6\times6}.
 \]
 The matrix above was introduced in \cite{GKP-Amen} which enabled the authors to conclude that the automorphism groups of ab-initio Hrushovski generic structures that are obtained from
 pre-dimension functions with rational coefficients are not amenable
 (Theorem 40 in \cite{GKP-Amen}). 
 
Definition \ref{def-conf} helps us to extend the negative result about
the amenability of automorphism groups of generic structures to those
that are obtained from pre-dimensions with irrational coefficients. 

Here, we prove the following
\begin{thm}
\label{thm:mat} Suppose $k\geq1$ is an integer. Then there is an
$n\times n$-matrix $\mathrm{X}$ whose entries are $0$ or $1$ and
$n>k$ such that $\mathrm{X}$ satisfies the $k$-configuration exhibiting
condition but does not satisfy the convex Ramsey property.
\end{thm}
\begin{proof}
We use the probabilistic method in combinatorics (see in \cite{MR3524748})
in order to show that such a matrix $\mathrm{X}$ exists. As we have
mentioned before for $k=1$ and $n=6$ existence of $\mathrm{X}$
is already proved in \cite{MR3524748} (and for $n\geq6$ a similar
idea can be modified). For the sake of simplicity we only present
the argument for $k=2$. At the end we mention what kind of change
is needed for $k>2$. 

Suppose $\mathrm{Y}$ is the following $n\times n$-matrix 
\[
\begin{pmatrix}1 & \cdots & 1 & 0 & \cdots & 0\\
1 & \cdots & 1 & 0 & \cdots & 0\\
1 & \cdots & 1 & 0 & \cdots & 0\\
1 & \cdots & 1 & 0 & \cdots & 0\\
1 & \cdots & 1 & 0 & \cdots & 0\\
1 & \cdots & 1 & 0 & \cdots & 0
\end{pmatrix}_{n\times n},
\]
where in each row the first $\frac{n}{2}$ entries are 1 and the rest
of the entries are 0. Now we construct matrix $\mathrm{X}$ from $\mathrm{Y}$
in the following random procedure. Independently, with probability
$p$ ($0<p<1$ and $q:=1-p>p$) change the entries of $\mathrm{Y}$:
if 1 to 0 and if 0 to 1. 

We first check the $2$-configuration condition of $\mathrm{X}$.
Write $\mbox{X}^{j}$ for the column $j$ of matrix $\mathrm{X}$
where $1\leq j\leq n$ and define $\mathrm{X}^{jk}:=\mathrm{X}^{j}\mathrm{X}^{k}$
where $1\leq j<k\leq n$. We now calculate the probability of the
failure of the $2$-configuration condition for $\mathrm{X}$. First
notice that the followings are true:
\begin{itemize}
\item $\mbox{Pr}\left(00\mbox{ occur in row \ensuremath{s} of }\mathrm{X}^{jk}\right)=\mbox{Pr}\left(11\mbox{ occur in row \ensuremath{s} of }\mathrm{X}^{lm}\right)=\mbox{Pr}\left(01\mbox{ occur in row \ensuremath{s} of }\mathrm{X}^{jl}\right)$
where $1\leq j,k\leq\frac{n}{2}<l,m\leq n$ and $1\leq s\leq n$.
\item $\mbox{Pr}\left(00\mbox{ occur in row \ensuremath{s} of }\mathrm{X}^{jk}\right)\leq\mbox{Pr}\left(\sigma\mbox{ occur in row \ensuremath{s} of }\mathrm{X}^{jk}\right)$
where $\sigma\in\left\{ 00,01,11,10\right\} $.
\end{itemize}
Then 
\[
\begin{array}{ccc}
\mbox{Pr}\left(\exists j,k\mbox{ s.t. }\mbox{2-config. fails for }\mbox{\ensuremath{\mathrm{X}}}^{jk}\right) & \leq & \sum_{j,k}\mbox{Pr}\left(\mbox{2-config. fails for }\mbox{\ensuremath{\mathrm{X}}}^{jk}\right)\\
 & \leq & \sum_{j,k}4\cdot\mbox{Pr}\left(00\mbox{ does not occur in any row of }\mbox{\ensuremath{\mathrm{X}}}^{jk}\right)\\
 & \leq & 4\cdot\PK n2\cdot\left(1-p^{2}\right)^{n}\\
 & \leq & 2\cdot n\cdot\left(n-1\right)\cdot e^{-p^{2}\cdot n}
\end{array}.
\]
Let $p:=n^{-\frac{1}{2}+\epsilon}$ for a fixed and small $\epsilon>0$.
Then $\mbox{Pr}\left(\exists j,k\mbox{ s.t. }\mbox{2-config. fails for }\mbox{\ensuremath{\mathrm{X}}}^{jk}\right)\rightarrow0$
as $n\rightarrow\infty$. Therefore, for large enough $n$ the random
matrix $\mathrm{X}$ satisfies the $2$-configuration condition almost
surely (i.e. with probability tending to 1 as $n\rightarrow\infty$). 

Now we want to prove that the convex Ramsey property eventually fails
for $\mathrm{X}$ when $n\rightarrow\infty$. Suppose on the contrary
that $\mathrm{X}$ satisfies the convex Ramsey property. Let $\mbox{\ensuremath{\mathrm{P}}}:=\begin{pmatrix}\mathsf{p}_{1} & \cdots & \mathsf{p}_{n}\end{pmatrix}_{1\times n}$
be a probability matrix such that
\[
\mathrm{P}\times\mathrm{X}\times\mathrm{W}\leq\frac{1}{2},
\]
for every Dirac-weight $n\times1$-matrix $\mathrm{W}$. Fix the following
notations: Define $\mathcal{I}^{j}:=\left\{ i:x_{ij}\neq y_{ij}\right\} $
for $1\leq j\leq n$. For $1\leq j\leq\frac{n}{2}$ write $\bar{j}:=j+\frac{n}{2}$
and define $\mathrm{W}^{j}$ be the Dirac-weight $n\times1$-matrix
that $w_{j,1}^{j}=1$ and $w_{\bar{j},1}^{j}=-1$. Then the convex
Ramsey property for $\mathrm{X}$ implies the following for $1\leq j\leq\frac{n}{2}$:
\[
\mathrm{P}\times\mathrm{X}\times\mathrm{W}^{j}=\sum_{i\notin\mathcal{I}^{j}\cup\mathcal{I}^{\bar{j}}}\mathsf{p}_{i}+\sum_{i\in\mathcal{I}^{j}\cup\mathcal{I}^{\bar{j}}}\xi_{i}\cdot\mathsf{p}_{i}\leq\frac{1}{2};
\]
where $\xi_{i}=\left\{ 1,0,-1\right\} $. Then $\sum_{i\in\mathcal{I}^{j}\cup\mathcal{I}^{\bar{j}}}\xi_{i}\cdot\mathsf{p}_{i}\geq\frac{1}{4}$
as $\sum_{i}\mathsf{p}_{i}=1$. Hence
\[
\sum_{1\leq j\leq\frac{n}{2}}\sum_{i\in\mathcal{I}^{j}\cup\mathcal{I}^{\bar{j}}}\xi_{i}\cdot\mathsf{p}_{i}\geq\frac{n}{2}\cdot\frac{1}{4}=\frac{n}{8}.
\]
Note that $\sum_{1\leq j\leq\frac{n}{2}}\sum_{i\in\mathcal{I}^{j}\cup\mathcal{I}^{\bar{j}}}\xi_{i}\cdot\mathsf{p}_{i}=\sum_{1\leq i\leq n}n_{i}\cdot\xi_{i}\cdot\mathsf{p}_{i}$
where $n_{i}:=\left\{ j:x_{ij}\neq y_{ij}\right\} $. 

Moreover, notice that $n_{i}$ has the binomial distribution $B\left(n,p\right)$,
and $p\leq\frac{n}{32}$ for large enough $n$. Thus using the Chernoff
bound one obtains the following 
\[
\begin{array}{ccc}
\sum_{1\leq i\leq n}\mbox{Pr}\left(n_{i}\geq\frac{n}{16}\right) & = & n\cdot\mbox{Pr}\left(n_{1}\geq\frac{n}{16}\right)\\
 & \leq & n\cdot\exp\left(-n\cdot D\left(\frac{n}{16}\parallel p\right)\right)\\
 & \leq & n\cdot\exp\left(-n\cdot D\left(\frac{n}{16}\parallel\frac{n}{32}\right)\right)\\
 & \leq & n\cdot e^{-c\cdot n}
\end{array}
\]
for some constant $c$ where $D\left(x\parallel y\right)=x\ln\frac{x}{y}+\left(1-x\right)\ln\left(\frac{1-x}{1-y}\right)$
is the Kullback-Leibler divergence. Note that $n\cdot e^{-c\cdot n}\rightarrow0$
when $n\rightarrow\infty$. Therefore, almost surely $n_{i}\leq\frac{n}{16}$
for all $i$ 
\[
\frac{n}{16}\geq\sum_{1\leq i\leq n}n_{i}\cdot\xi_{i}\cdot\mathsf{p}_{i}=\sum_{1\leq j\leq\frac{n}{2}}\sum_{i\in\mathcal{I}^{j}\cup\mathcal{I}^{\bar{j}}}\xi_{i}\cdot\mathsf{p}_{i}\geq\frac{n}{8},
\]
which is a contradiction. Hence, the randomly generated matrix $\mathrm{X}$
almost surely does not satisfy the convex Ramsey condition for large
enough $n$.

For $k>2$ a similar argument works and we only need to consider $p=n^{-\frac{1}{k}+\epsilon}$
this time.
\end{proof}

\section{Generic structures with embeddings that form free-pseudoplanes }
\label{sec:main}

Similar to Section 4 in \cite{GKP-Amen}, using Theorem 32 in \cite{GKP-Amen}
we present a sufficient condition for two embeddings to show that the
automorphism groups of certain generic structures are not amenable.
The condition is called forming a \emph{free-pseudoplane} and automorphism groups
of generic structures with such embeddings include especially
the automorphism groups of generic structures that are obtained from
pre-dimensions with irrational coefficients (see Subsection $\ref{subsec: pre-dim}$).
\begin{defn}
\label{def:imp} Let $\left(\mathcal{C},\leqslant\right)$ be a smooth
class with AP, and let $\mathcal{M}$ be the $\left(\mathcal{C},\leqslant\right)$-generic
structure. Suppose $X\subseteq\mathcal{M}$ and let $A,B\in\mathcal{C}$
with $A\leqslant B$. Suppose $k\geq2$ is an integer such that $k<\left|\PK BA\right|$. 
\begin{enumerate}
\item We call embeddings $\Lambda\in\PK XB$ a \emph{$B$-line in $X$ }and
embedding $\Gamma\in\PK XA$ an \emph{$A$-point in $X$.} 
\item We say two $B$-lines $\Lambda^1$ and $\Lambda^2$
in $X$ are \emph{connected via} \emph{a path} in $X$ if there are
$d\geq1$ and $B$-lines $\Lambda_{i_{0}},\cdots,\Lambda_{i_{d}}$
in $X$ such that $\Lambda_{i_{0}}=\Lambda^1$,
$\Lambda_{i_{d}}=\Lambda^2$, and $\Lambda_{i_{j}}\left(B\right)\cap\Lambda_{i_{j+1}}\left(B\right)$
contains at least one $A$-point (in $X$) for each $0\leq j<d$.
We say $\Lambda^1$ and $\Lambda^2$
have \emph{distance} $d$ if $d+1$ is the minimum number of embeddings
needed to connect $\Lambda^1$ and $\Lambda^2$
via a path. 
\item We say $\Lambda^1$ and $\Lambda^2$
lay on an $m$-cycle (for $m\geq2$) of $B$-lines in $X$ if there
are distinct $B$-lines $\Lambda_{i_{0}},\cdots,\Lambda_{i_{m-1}}$
in $X$ where $\Lambda_{i_{0}}=\Lambda^1$, $\Lambda_{i_{m-1}}=\Lambda^2$
such that:
\begin{enumerate}
\item $\Lambda_{i_{0}}\left(B\right)\cap\Lambda_{i_{m-1}}\left(B\right)$
and $\Lambda_{i_{j}}\left(B\right)\cap\Lambda_{i_{j+1}}\left(B\right)$
contain at least one $A$-point, for each $0\leq j<m-1$; 
\item The intersection of any two distinct elements of the set  \\ $\left\{ \Lambda_{i_{0}}\left(B\right)\cap\Lambda_{i_{m-1}}\left(B\right),\Lambda_{i_{j}}\left(B\right)\cap\Lambda_{i_{j+1}}\left(B\right):0\leq j<m-1\right\} $ does not
contain a common $A$-point. 
\end{enumerate}
\item We say $\left(A;B\right)$ forms a \emph{$k$-pseudoplane} in $X$
if every two distinct $B$-lines in $X$ intersect in at most $\left(k-1\right)$-many
$A$-points. 
\item Suppose $\Lambda$ is a $B$-line in $X$. Let
\[
\mathcal{I}_{X}\left(\Lambda\right):=\left\{ \Gamma\left(A\right):\Gamma\in\PK{\mathcal{M}}A,\Gamma\left(A\right)\subseteq\Lambda\left(B\right),\exists\Lambda'\neq\Lambda,\Lambda'\left(B\right)\subseteq X,\Gamma\left(A\right)\subseteq\Lambda'\left(B\right)\right\} ,
\]
and $\mathsf{I}_{X}\left(\text{\ensuremath{\Lambda}}\right):=\left|\mathcal{I}_{X}\left(\Lambda\right)\right|$
(i.e. $\mathcal{I}_{X}\left(\Lambda\right)$ is the set of all $A$-points
in $X$ that lay in strictly more than one $B$-point in $X$). Define
$\mathcal{H}_{k}^{\Lambda}\left(X\right):=\Lambda\left(B\right)\backslash\bigcup\mathcal{I}_{X}\left(\Lambda\right)$
when $\mathsf{I}_{X}\left(\text{\ensuremath{\Lambda}}\right)\leq k$;
otherwise let $\mathcal{H}_{k}^{\Lambda}\left(X\right)=\emptyset$.
\item Suppose $X\subseteq\mathcal{M}$ is a finite substructure and let
$\left(\Lambda_{i}:0\leq i<b\right)$ be an enumeration of all $B$-lines
in $X$ where $b=\left|\PK XB\right|$. Let $X_{0}:=X$ and define
inductively $X_{i+1}:=X_{i}\backslash\mathcal{H}_{k}^{\Lambda_{i}}\left(X_{i}\right)$
for $1\leq i<b$. We say \emph{$\left(A;B\right)$ }forms a \emph{free}
\emph{$k$-}pseudoplane in $X$ if there is an enumeration $\left(\Lambda_{i}:i\in b\right)$
of all $B$-lines in $X$ such that $\PK{X_{b}}A=\emptyset$. We say
\emph{$\left(A;B\right)$ }forms a\emph{ }free \emph{$k$-}pseudoplane
in an infinite subset $X\subseteq\mathcal{M}$ if \emph{$\left(A;B\right)$
}form a\emph{ }free \emph{$k$-}pseudoplane in every finite subset
of $X$.
\end{enumerate}
\end{defn}
\begin{rem}
\label{rem:k-free}
\begin{enumerate}
\item Using this new terminology a tree-pair (Definition 39 in \cite{GKP-Amen})
is a free $2$-pseudoplane. However, the converse is not true.
\item Note that for any finite substructure $X\subseteq\mathcal{M}$ and
any enumeration $\left(\Lambda_{i}:0\leq i<b\right)$ of all $B$-lines
in $X$, when $i\rightarrow\infty$ the set $X_{i+1}:=X_{i}\backslash\mathcal{H}_{k}^{\Lambda_{i^{*}}}\left(X_{i}\right)$
eventually remains fixed where $i^{*}\equiv^{b}i$ and $0\leq i^{*}\leq b$
. If \emph{$\left(A;B\right)$ }forms a\emph{ }free \emph{$k$-}pseudoplane
in $\mathcal{M}$, then for every enumeration that we choose there
is $i_{0}$ such that $\PK{X_{i_{0}}}A=\emptyset$.
\end{enumerate}
\end{rem}
Then we can prove the following which is similar to Theorem 40 in \cite{GKP-Amen}.
\begin{thm}
\label{thm:2-pseudo} Suppose $\left(\mathcal{C},\leqslant\right)$
is a smooth class with AP, and $\mathcal{M}$ is the $\left(\mathcal{C},\leqslant\right)$-generic
structure. Suppose for $n\in\mathbb{N}$ there is $A,B\in\mathcal{C}$
such that $\left|\PK BA\right|=n$ and $\left(A;B\right)$ is a free
$2$-pseudoplane. Moreover assume there is an $n\times m$-matrix
$\mathrm{X}$ such that $\mathrm{X}$ satisfies 2-configuration exhibiting
condition but does not satisfy the convex Ramsey property. Then $\Aut\left(\mathcal{M}\right)$
is not amenable.
\end{thm}
\begin{proof}
We present a coloring function $f:\PK{\mathcal{M}}A\rightarrow\left\{ 0,1\right\} $
such that the full-coloring matrix of $f$ for copies of $B$ in $\mathcal{M}$
is the matrix $\mathrm{X}$. We first prove that one can assign a
\emph{consistent} coloring for every finite subset of $\mathcal{M}$
using rows of $\mathrm{X}$. 

Since $\left(A;B\right)$ forms a free 2-pseudoplane then for every
finite substructure $C$ of $M$ there is an enumeration $\left(\Lambda_{i}:0\leq i<c\right)$
of embeddings of $B$-lines in $C$ such that $C_{0}:=C$ and $\PK{C_{c}}A=\emptyset$
where $c=\left|\PK CB\right|$ and $C_{i+1}:=C_{i}\backslash\mathcal{H}_{k}^{\Lambda_{i}}\left(C_{i}\right)$
is inductively defined. We consistently color $B$-lines using rows
of $\mathrm{X}$, inductively. Start with $\Lambda_{c-1}$. Pick a
row in the matrix $\mathrm{X}$ and assign a coloring for $A$-points
of $\Lambda_{c-1}\left(B\right)$ according to the row. Now suppose
a consistent coloring for $\Lambda_{i}\left(B\right)$ is already
chosen according to the rows of the matrix $\mathrm{X}$. Then $\Lambda_{i-1}\left(B\right)$
interests with $A$-lines of $C_{i}$ in at most $2$ many $A$-points.
Since the matrix $\mathrm{X}$ satisfies the 2-configuration exhibiting
condition then there is a row in $\mathrm{X}$ that the colorings
of $A$-points agree with the coloring of $A$-points in the intersection.
Hence, we can pick a coloring for $\Lambda_{i-1}$ from rows of $\mathrm{X}$
in a consistent way and therefore one can use matrix $\mathrm{X}$
to color all $B$-points of $C$.

The matrix $\mathrm{X}$ have only finitely many rows and by the above
argument for every finite subset of $\mathcal{M}$ there is a consistent
coloring using $\mathrm{X}$. Then by Rado's selection lemma in \cite{MR0270920}
there is a coloring $f$ of $A$-points of $\mathcal{M}$ that its
coloring matrix of $B$-lines in $\mathcal{M}$ is $\mathrm{X}$.\footnote{The second author would like to thank Martin Ziegler for suggesting
this shorter proof.} Then Theorem 32 in \cite{GKP-Amen} implies that $\Aut\left(\mathcal{M}\right)$
does not have the convex $\leqslant$-Ramsey property with respect
to $\left(\mathcal{C},\leqslant\right)$ and hence not $\Aut\left(\mathcal{M}\right)$
is not amenable. 
\end{proof}
Similarly, one can prove the following
\begin{thm}
\label{thm:2-pseudo-1} Suppose $\left(\mathcal{C},\leqslant\right)$
is a smooth class with AP, and $\mathcal{M}$ is the $\left(\mathcal{C},\leqslant\right)$-generic
structure. Assume for $n\in\mathbb{N}$ there are $A,B\in\mathcal{C}$
such that $\left|\PK BA\right|=n$ and $\left(A;B\right)$ is a free
k-pseudoplane and there is an $n\times m$-matrix $\mathrm{X}$ such
that $\mathrm{X}$ satisfies k-configuration exhibiting condition
but does not satisfy the convex Ramsey property. Then $\Aut\left(\mathcal{M}\right)$
is not amenable.
\end{thm}

\subsection{Pre-dimensions with irrational coefficients. }

\label{subsec: pre-dim}

Let $\mathcal{\mathbf{K}}$ be the class of all finite graphs and
$\alpha\in\left(0,1\right)\backslash\mathbb{Q}$. Define $\delta_{\alpha}:\mathcal{\mathcal{\mathbf{K}}}\longrightarrow\mathbb{R}$
as $\delta_{\alpha}\left(A\right)=\left|A\right|-\alpha\cdot\left|\mathfrak{R}\left(A\right)\right|$
where $\mathfrak{R}\left(A\right)$ is the set $\mathfrak{R}$-relations
of $A$ (set of all edges of $A$). For every $A\subseteq B\in\mathcal{\mathbf{K}}$,
define $A\leqslant_{\alpha}B$ if and only if $\delta_{\alpha}\left(C\right)-\delta_{\alpha}\left(A\right)\geq0,$
for every $C$ with $A\subseteq C\subseteq B$. Finally, put $\mathcal{\mathbf{K}}_{\alpha}:=\left\{ A\in\mathcal{\mathbf{K}}:\delta_{\alpha}\left(A'\right)\geq0,\mbox{ for every }A'\subseteq A\right\} $. The class $\mathcal{\mathbf{K}}_{\alpha}$ contains the class of \emph{sparse graphs}.
\begin{notation}
Suppose $A,B,C\in\mathcal{\mathbf{K}}_{\alpha}$ and $A,B\subseteq C$.
Then let 
\[
\delta_{\alpha}\left(A/B\right):=\delta_{\alpha}\left(AB\right)-\delta_{\alpha}\left(B\right).
\]
\end{notation}
\begin{fact}
(See \cite{Balshi}) The $\left(\mathcal{\mathbf{K}}_{\alpha},\leqslant_{\alpha}\right)$
is a smooth class with the free-amalgamation property. 
\end{fact}
We call $\left(\mathcal{\mathbf{K}}_{\alpha},\leqslant_{\alpha}\right)$
an \emph{ab-initio }smooth class that is obtained from $\delta_{\alpha}$
and we write $\mathcal{\mathbf{M}}^{\alpha}$ for the countable $\left(\mathcal{\mathbf{K}}_{\alpha},\leqslant_{\alpha}\right)$-generic
structures that is obtained from Theorem \ref{thm:2-pseudo-1}. As we mentioned in the introduction, the theory  $\mathcal{\mathbf{M}}^{\alpha}$ 
is the zero-one law theory of graphs with the edge probability $n^{-\alpha}$
(see \cite{MR924703}).

Here are some important properties of $\delta_{\alpha}$
\begin{fact}
(See \cite{Balshi}) Suppose $A,B,C\subseteq\mathbf{M}^{\alpha}$
are finite subsets. Then the followings hold 
\begin{enumerate}
\item $\delta_{\alpha}\left(ABC\right)=\delta_{\alpha}\left(AB/C\right)+\delta_{\alpha}\left(C\right)=\delta_{\alpha}\left(A/BC\right)+\delta_{\alpha}\left(B/C\right)+\delta_{\alpha}\left(C\right).$
\item $\delta_{\alpha}(AB/C)\leq\delta_{\alpha}(A/C)+\delta_{\alpha}(B/C)-\delta_{\alpha}\left(\left(A\cap B\right)/C\right)$. 
\item Assume $A,B$ are disjoint and $\neg\mathfrak{R}^{\mathbf{M}^{\alpha}}\left(a,b\right)$
for every $a\in A$ and $b\in B$. Then $\delta_{\alpha}(AB/C)=\delta_{\alpha}(A/C)+\delta_{\alpha}(B/C)$. 
\end{enumerate}
\end{fact}
\begin{lem}
Suppose $\left(A;B\right)$ forms a 2-pseudoplane in $\mathcal{\mathbf{M}}^{\alpha}$.
If $\delta_{\alpha}\left(B/A\right)<\frac{\delta_{\alpha}\left(A\right)}{2}$,
then $\left(A;B\right)$ forms a free 2-pseudoplane.
\end{lem}
\begin{proof}
Put $\varepsilon:=\delta_{\alpha}\left(B/A\right)$. Let $X$ be a
finite subset of $\mathcal{\mathbf{M}}^{\alpha}$. Let $\left(\Lambda_{i}:i\in I\right)$
be an enumeration. Let sets $X_{i}\subseteq X$ for $i\in\mathbb{N}$
be defined similar to Remark $\ref{rem:k-free}$-(2). Let $i_{0}$
be the minimum number such that $X_{i_{0}}=X_{i_{0}+1}$. If $\PK{X_{i_{0}}}A=\emptyset$,
then it means that $X$ is a free 2-pseudoplane. Otherwise, let $m$
be the number of $B$-lines in $X_{i_{0}}$ . Then 
\[
\delta_{\alpha}\left(B\right)\leq\delta_{\alpha}\left(X_{i_{0}}\right)\leq m\cdot\delta_{\alpha}\left(B\right)-\frac{3m}{2}\cdot\delta_{\alpha}\left(A\right),
\]
where $\frac{3m}{2}$ is the minimum number of recalculating $\delta_{\alpha}\left(A\right)$.
Then 
\[
\begin{array}{ccc}
0 & \leq & 2\left(m-1\right)\cdot\left(\delta_{\alpha}\left(B/A\right)\right)-\left(m+2\right)\cdot\delta_{\alpha}\left(A\right)\\
\left(m+2\right)\cdot\delta_{\alpha}\left(A\right) & \leq & 2\left(m-1\right).\varepsilon\\
\frac{m+2}{2\left(m-1\right)}\cdot\delta_{\alpha}\left(A\right) & \leq & \varepsilon
\end{array}
\]
It is clear that $\frac{m+2}{2\left(m-1\right)}$ is decreasing when
$m\rightarrow\infty$. Hence $\frac{\delta_{\alpha}\left(A\right)}{2}\leq\varepsilon$
which is a contraction with our assumption that $\varepsilon<\frac{\delta_{\alpha}\left(A\right)}{2}$.
Therefore $\PK{X_{i_{0}}}A=\emptyset$ and $X$ is free 2-pseudoplane.
\end{proof}
\begin{lem}
\label{lem:main} For $n\geq3$, there are $A,B\in\mathbf{K}_{\alpha}$
such that
\begin{enumerate}
\item $A\leqslant B$ and $\left|\PK BA\right|=n$;
\item $\left(A;B\right)$ is a free $2$-pseudoplane.
\end{enumerate}
\end{lem}
Before proving Lemma $\ref{lem:main}$, we prove some technical lemmas
(but folklore) that are used in the proof. 
\begin{fact}
\label{fact:cofinality} For an irrational $\alpha\in\left(0,1\right)$,
for every $N\in\mathbb{N}$ there are infinitely many integers $r_{i}$
and $s_{i}$ such that $-\frac{1}{N}<r_{i}-\alpha\cdot s_{i}<0$.
\end{fact}
Write $K_{n}$ for the complete graph with $n$-vertices and $K_{n.m}$
for the complete $\left(n,m\right)$-bipartite graph.
\begin{lem}
\label{lem:small} For every $N\in\mathbb{N}$, there is $EF\in\mathbf{K}_{\alpha}$
such that 
\begin{enumerate}
\item $E\cap F=\emptyset$ and $-\frac{1}{N}<\delta_{\alpha}\left(E/F\right)<0$;
\item Every $\leqslant$-closed $K_{3}$-embedding in $EF$ is a $\leqslant$-closed
$K_{3}$-embedding in $F$. 
\end{enumerate}
\end{lem}
\begin{proof}
Let $m$ be an integer such that $\frac{1}{m+1}\leq\alpha\leq\frac{1}{m}$
and using Fact $\ref{fact:cofinality}$, choose $r,s\in\mathbb{N}$
such that $r>m^{2}+\frac{m+1}{N}$ and $-\frac{1}{N}<r-\alpha\cdot s<0$.
Then $s<\frac{1}{\alpha}\cdot\left(r+\frac{1}{N}\right)\leq\left(m+1\right)\cdot\left(r+\frac{1}{N}\right)$.
Let $r_{0}:=r-m$ and $r_{1}:=m$. Consider the complete bipartite
graph $K_{r_{0},r_{1}}$. Then 
\[
0\leq r-\frac{r_{0}r_{1}}{m}\leq r-\alpha\cdot r_{0}r_{1}=\delta_{\alpha}\left(K_{r_{0},r_{1}}\right).
\]
Then it is easy also to see $s-r_{0}r_{1}=s-m\left(r-m\right)\leq\left(m+1\right)\cdot\left(r+\frac{1}{N}\right)-mr+m^{2}=r+\frac{m+1}{N}+m^{2}$.
Now we introduce a graph $EF$ that satisfies the properties of the
lemma. Let $E$ be a graph that $E\cong K_{r_{0},r_{1}}$, and take
$F\in\mathcal{K_{\alpha}}$ with $4\leq\left|F\right|\leq s-r_{0}r_{1}$
such that $\delta\left(F'\right)\geq\frac{1}{N}$ for every nonempty
$F'\subseteq F$ and moreover there are two pairs of vertices that
are not connected (for example any graph without edges). We want to
draw $\left(s-r_{0}r_{1}\right)$-many edges between $E$ and $F$
in such way that no $K_{3}$-graphs (or 3-cycles) appears in $EF$
apart from possibly those in $F$. By our assumption $F$ contains
two pair of vertices $\left(f_{k_{1}},f_{k_{2}}\right)$ and $\left(f_{k_{3}},f_{k_{4}}\right)$
that are not connected (i.e. $\neg\mathfrak{R}^{F}\left(f_{k_{1}},f_{k_{3}}\right)\wedge\neg\mathfrak{R}^{F}\left(f_{k_{2}},f_{k_{4}}\right)$).
Note that $E$ is a bipartite graph and let $E_{1}$ and $E_{2}$
be the bipartite partitions of $E$ and let $\left(e_{i}^{1}:i<r_{0}\right)$
and $\left(e_{i}^{2}:i<r_{1}\right)$ be the enumeration of elements
of $E_{1}$ and $E_{2}$; respectively. Now let
\begin{enumerate}
\item $EF\models\bigwedge_{i<r_{0}}\mathfrak{R}\left(e_{i}^{1},f_{k_{1}}\right)\wedge\bigwedge_{i<r_{1}}\mathfrak{R}\left(e_{i}^{2},f_{k_{2}}\right)$;
\item $EF\models\bigwedge_{i<m_{0}}\mathfrak{R}\left(e_{i}^{1},f_{k_{3}}\right)\wedge\bigwedge_{i<m_{1}}\mathfrak{R}\left(e_{i}^{2},f_{k_{4}}\right)$;
\end{enumerate}
where $m_{0}+m_{1}=s-r_{0}r_{1}$ and $m_{0}\leq r_{0}$ and $m_{1}\leq r_{1}$.
Note that valency of each vertice in $E$ is at least $\left(m+1\right)$
and one can check that that $\delta_{\alpha}\left(E'/F\right)\geq\delta_{\alpha}\left(E/F\right)$
for every $E'\subseteq E$. Here are more detailed calculations: suppose
$m^{*}\leq m$ and $r^{*}\leq r_{0}$ where $m^{*}+r^{*}=\left|E'\right|$.
Then 
\[
\begin{array}{ccc}
\delta_{\alpha}\left(E/E'F\right) & \leq & \left(r_{0}-r^{*}\right)+\left(m-m^{*}\right)-\left(\left(m+1\right)\cdot\left(r_{0}-r^{*}\right)+\left(m-m^{*}\right)\cdot\left(r^{*}+1\right)\right)\cdot\alpha\\
 & \leq & \left(r_{0}-r^{*}\right)+\left(m-m^{*}\right)-\left(\left(m+1\right)\cdot\left(r_{0}-r^{*}\right)+\left(m-m^{*}\right)\cdot\left(r^{*}+1\right)\right)\cdot\frac{1}{m+1}\\
 & \leq & \left(m-m^{*}\right)\left(1-\frac{r^{*}+1}{m+1}\right)
\end{array}
\]
When $r^{*}>m$ then $\delta_{\alpha}\left(E/E'F\right)<0$. Note
that $\delta_{\alpha}\left(E/F\right)=\delta_{\alpha}\left(E/E'F\right)+\delta_{\alpha}\left(E'/F\right)$
and therefore $\delta_{\alpha}\left(E'/F\right)\geq\delta_{\alpha}\left(E/F\right)$.
When $r^{*}\leq m$ then 
\[
\begin{array}{ccc}
\delta_{\alpha}\left(E'/F\right) & = & r^{*}+m^{*}-\left(r^{*}m^{*}+r^{*}+m^{*}\right)\cdot\alpha\\
 & \geq & r^{*}+m^{*}-\frac{r^{*}m^{*}+r^{*}+m^{*}}{m}\\
 & \geq & \frac{\left(m-1\right)\cdot\left(r^{*}+m^{*}\right)-r^{*}m^{*}}{m}\\
 & \geq & \frac{\left(r^{*}+m^{*}\right)\cdot\left(4\left(m-1\right)-\left(r^{*}+m^{*}\right)\right)}{4m}\\
 & \geq & 0
\end{array}
\]
The last two inequality holds because $r^{*}+m^{*}\leq2m$ and $r^{*}m^{*}\leq\left(\frac{r^{*}+m^{*}}{2}\right)^{2}$.
Therefore, in order to check that $EF\in\mathbf{K}_{\alpha}$, we
only need to check $\delta_{\alpha}\left(EF'\right)\geq0$ for every
$F'\subseteq f_{k_{1}}f_{k_{2}}f_{k_{3}}f_{k_{4}}$. By the properties
of the pre-dimension $\delta_{\alpha}\left(EF'\right)=\delta_{\alpha}\left(E/F'\right)+\delta_{\alpha}\left(F'\right)>-\frac{1}{N}+\frac{1}{N}\geq0$
and hence $EF\in\mathbf{K}_{\alpha}$. Notice that all the $K_{3}$-embeddings
of $EF$ have to be $F$.
\end{proof}

\begin{proof}
[Proof of Lemma 2] Fix $n$. We construct $A,B\in\mathbf{K}_{\alpha}$
using the free-amalgamation property of $\mathbf{K}_{\alpha}$. Let
$2\leq i\leq n$ and write $\left[n\right]^{i}:=\left\{ u\subseteq\left\{ 0,\cdots,n-1\right\} :\left|u\right|=i\right\} $.
Let $A$ be the $K_{3}$-complete graph; namely, $A$ is a graph with
three vertices $a_{i}$ with $0\leq i<3$ such that $A\models\mathfrak{R}\left(a_{0},a_{1}\right)\wedge\mathfrak{R}\left(a_{1},a_{2}\right)\wedge\mathfrak{R}\left(a_{2},a_{0}\right)$.
Then $\delta_{\alpha}\left(A\right)=3-3\cdot\alpha\geq0$ as $\alpha\in\left(0,1\right)$
and hence $A\in\mathbf{K}_{\alpha}$. We construct $B\in\mathbf{K}_{\alpha}$
such that \textbf{$B:=\dot{\bigcup}_{i\in n}A_{i}\dot{\cup}\dot{\bigcup}_{u\in\left[n\right]^{i},2\leq i\leq n}X_{u}$}
as a set, where $A_{i}$'s are isomorphic copies of $A$ with the
following properties:
\begin{enumerate}
\item $\delta_{\alpha}\left(B\right)\geq\delta_{\alpha}\left(A_{i}\right)$
for $i\in n$;
\item $\delta_{\alpha}\left(C\right)>\delta_{\alpha}\left(B\right)$ for
every $C\subsetneqq B$ that contains at least two copies of $A$,
namely $\cl C=B$;
\item $\left|\PK BA\right|=n$.
\end{enumerate}
It is clear if Conditions (1),(2) and (3) hold for $A$ and $B$,
then $\left(A;B\right)$ is a 2-pseudoplane. For Condition $\left(3\right)$
we need to make sure that there are no $\mathit{K}_{3}$-graphs in
$B$ apart from $A_{i}$'s. In order to satisfy Condition (2) for
every $2\leq i\leq n$ and $u=\left\{ u_{0},\cdots,u_{i-1}\right\} \in\left[n\right]^{i}$
we construct $X_{u}$ such that $\delta_{\alpha}\left(X_{u}/\left(\bigcup_{j\in i}A_{u_{j}}\right)\right)<0$.
We now check the conditions for $X_{u}$ in order to get the properties
that are needed. Let $\epsilon_{u}:=-\delta_{\alpha}\left(X_{u}/\left(\bigcup_{j\in i}A_{u_{j}}\right)\right)$.
Put $\epsilon_{0}:=\text{min}_{2\leq i\leq n}\left\{ \epsilon_{u}:u\in\left[n\right]^{i}\right\} $
and $\epsilon_{1}:=\text{max}_{2\leq i\leq n}\left\{ \epsilon_{u}:u\in\left[n\right]^{i}\right\} $
which are non-zero. Here are the calculations for finding sufficient
conditions
\[
n\cdot\delta_{\alpha}\left(A\right)-\left(2^{n}-n-1\right)\cdot\epsilon_{1}\leq\delta_{\alpha}\left(B\right)=n\cdot\delta_{\alpha}\left(A\right)-\sum_{2\leq i\leq n}\sum_{u\in\left[n\right]^{i}}\epsilon_{u}\leq n\cdot\delta_{\alpha}\left(A\right)-\left(2^{n}-n-1\right)\cdot\epsilon_{0}.
\]
 In order to have $\left(A;B\right)$ satisfying Condition (1) we
demand
\[
\begin{array}{ccc}
n\cdot\delta_{\alpha}\left(A\right)-\left(2^{n}-n-1\right)\cdot\epsilon_{1} & \geq & \delta_{\alpha}\left(A\right)\\
\left(n-1\right)\cdot\delta_{\alpha}\left(A\right) & \geq & \left(2^{n}-n-1\right)\cdot\epsilon_{1}\\
\frac{n-1}{\left(2^{n}-n-1\right)}\cdot\delta_{\alpha}\left(A\right) & \geq & \epsilon_{1}
\end{array}
\]
Suppose $2\leq k<n$ and let $c\in\left[n\right]^{k}$. Then 
\[
k\cdot\delta_{\alpha}\left(A\right)-\left(2^{k}-k-1\right)\cdot\epsilon_{1}\leq\delta_{\alpha}\left(\bigcup_{i\in c}A_{i}\cup\bigcup_{2\leq i\leq n}\bigcup_{u\in\left[n\right]^{i},u\subset c}X_{u}\right)=k\cdot\delta_{\alpha}\left(A\right)-\sum_{2\leq i\leq n}\sum_{u\in\left[n\right]^{i},u\subseteq c}\epsilon_{u}.
\]
We demand the following in order to obtain Condition (2)
\[
\begin{array}{ccc}
k\cdot\delta_{\alpha}\left(A\right)-\left(2^{k}-k-1\right)\cdot\epsilon_{1} & > & \delta_{\alpha}\left(B\right)\end{array}.
\]
Note that $k\cdot\delta_{\alpha}\left(A\right)-\left(2^{k}-k-1\right)\cdot\epsilon_{1}$
is increasing as $k\rightarrow n-1$. Hence, we only need to satisfy:

\[
\begin{array}{ccc}
2\cdot\delta_{\alpha}\left(A\right)-\left(2^{2}-2-1\right)\cdot\epsilon_{1} & > & \delta_{\alpha}\left(B\right)\end{array}.
\]
Then we demand

\[
\begin{array}{ccc}
2\cdot\delta_{\alpha}\left(A\right)-\epsilon_{1} & \geq & n\cdot\delta_{\alpha}\left(A\right)-\left(2^{n}-n-1\right)\cdot\epsilon_{0}\\
\left(2^{n}-n-1\right)\cdot\epsilon_{0}-\epsilon_{1} & \geq & \left(n-2\right)\cdot\delta_{\alpha}\left(A\right)\\
\epsilon_{0} & \geq & \frac{n-2}{2^{n}-n-2}\cdot\delta_{\alpha}\left(A\right)+\xi
\end{array}
\]
where $\xi=\frac{\epsilon_{1}-\epsilon_{0}}{2^{n-1}-n-2}$. Then all
these conditions together demand 

\[
\begin{array}{ccc}
\xi+\frac{n-2}{2^{n}-n-2}\cdot\delta_{\alpha}\left(A\right)\leq\epsilon_{0}\leq\epsilon_{u}\leq\epsilon_{1}\leq\frac{n-1}{2^{n}-n-1}\cdot\delta_{\alpha}\left(A\right) &  & \left(*\right)\end{array}
\]
for every $u\in\left[n\right]^{i}$ where $2\leq i\leq n$. One can
check that when $n\geq3$ the inequalities above are valid. We choose
$\epsilon_{u}$'s in such a way that $\xi$ is ignorable
\[
\frac{n-1}{\left(2^{n}-n-1\right)}\cdot\delta_{\alpha}\left(A\right)-\epsilon_{u}\leq\frac{n-1}{\left(2^{n}-n-1\right)}\cdot\delta_{\alpha}\left(A\right)-\epsilon_{0}\leq\frac{1}{C}
\]
for some $C\in\mathbb{N}$. Now we use Lemma $\ref{lem:small}$ and
Fact $\ref{fact:cofinality}$ we construct $X_{u}$ such that $\epsilon_{u}=\delta_{\alpha}\left(X_{u}/\left(\bigcup_{j\in i}A_{u_{j}}\right)\right)$
satisfies $\left(*\right)$. Using Fact $\ref{fact:cofinality}$ for
every $u\in\left[n\right]^{i}$ choose $r_{u},s_{u}$ such that $-\frac{n-1}{\left(2^{n}-n-1\right)}\cdot\delta_{\alpha}\left(A\right)<r_{u}-\alpha\cdot s_{u}\leq-\left(\frac{n-1}{\left(2^{n}-n-1\right)}-\frac{1}{C}\right)\cdot\delta_{\alpha}\left(A\right)$
and let $\epsilon_{u}:=-\left(r_{u}-\alpha\cdot s_{u}\right)$. Moreover,
we choose $r_{u}$ and $s_{u}$'s such that $\xi$ is ignorable and
the followings hold
\begin{enumerate}
\item $r_{u}-s_{u}\cdot\alpha=-\epsilon_{u}<0$;
\item $r_{u}-\left(s_{u}-r_{u}\right)\alpha=-\epsilon_{u}+r_{u}\cdot\alpha\geq0$.
\end{enumerate}
Note that Condition ($2$) is satisfiable because $r_{u}$ can be
chosen arbitrarily large. We construct $X_{u}$ every $u\in\left[n\right]^{i}$
such that $\left|X_{u}\right|=r_{u}$, $\delta_{\alpha}\left(X_{u}/\left(\bigcup_{j\in i}A_{u_{j}}\right)\right)=-\epsilon_{u}$
and $X_{u}\cup\left(\bigcup_{j\in i}A_{u_{j}}\right)$ has only $i$-many
$K_{3}$-subgraphs, namely only $A_{u_{i}}$'s. Using the same idea
that is used in the proof of Lemma $\ref{lem:small}$ one can find
$K_{3}$-free graphs $E$ of arbitrary the required size in such that
$\delta_{\alpha}\left(E/\left(\bigcup_{j\in i}A_{u_{j}}\right)\right)<0$.
By modifying the edges drawn between $E$ and $F=\bigcup_{j\in i}A_{u_{j}}$
in Lemma $\ref{lem:small}$, we can choose $X_{u}\cong E$ to be a
bipartite graph of size $r_{u}$ such that $X_{u}\in\mathbf{K}_{\alpha}$
and $\left|\mathfrak{R}\left(X_{u}\right)\right|=s_{u}-r_{u}$. Note
that $A_{u_{j}}$'s for $j\in i$ are disjoint. Note that $F$ satisfies
the condition of Lemma $\ref{lem:small}$. We modify the drawn edges
between $E$ and $F$ is such a that for each $j\in i$ there is at
least one edges between $X_{u}$ and $A_{u_{j}}$. It is easy draw
the edges in such way that no $K_{3}$-graph is embeddable in $X_{u}\cup\left(\bigcup_{j\in i}A_{u_{j}}\right)$
apart from $\bigcup_{j\in i}A_{u_{j}}$. Hence $\left|\PK{X_{u}\cup\left(\bigcup_{j\in i}A_{u_{j}}\right)}A\right|=i$.
Now let $B$ be the free-amalgamation of all these $X_{u}\cup\left(\bigcup_{j\in i}A_{u_{j}}\right)$
over $\dot{\bigcup}_{l\in n}A_{l}$. Then $\left|\PK BA\right|=n$
and Conditions (1-3) hold for $B$ and $A$ hence they are $2$-pseudoplane.
We now want to show embeddings of $B$ and $A$ form free 2-pseudoplane.
By Lemma we need to check whether $\delta_{\alpha}\left(B/A\right)=\delta_{\alpha}\left(B\right)-\delta_{\alpha}\left(A\right)<\frac{1}{2}\cdot\delta_{\alpha}\left(A\right)$
holds. As we have chosen 

\[
\left(\frac{n-1}{\left(2^{n}-n-1\right)}-\frac{1}{C}\right)\cdot\delta_{\alpha}\left(A\right)\leq\varepsilon_{u}\leq\frac{n-1}{\left(2^{n}-n-1\right)}\cdot\delta_{\alpha}\left(A\right)
\]
Then 
\[
\begin{array}{ccc}
2\cdot\delta_{\alpha}\left(B\right)-3\cdot\delta_{\alpha}\left(A\right) & < & 2\left(n\cdot\delta_{\alpha}\left(A\right)-\left(2^{n}-n-1\right)\cdot\left(\frac{n-1}{\left(2^{n}-n-1\right)}-\frac{1}{C}\right)\cdot\delta_{\alpha}\left(A\right)\right)-3\cdot\delta_{\alpha}\left(A\right)\\
 & < & 2\delta_{\alpha}\left(A\right)-3\delta_{\alpha}\left(A\right)+2\cdot\frac{2^{n}-n-1}{C}\cdot\delta_{\alpha}\left(A\right)\\
 & < & 2\cdot\frac{2^{n}-n-1}{C}\cdot\delta_{\alpha}\left(A\right)-\delta_{\alpha}\left(A\right)
\end{array}
\]
If we choose $C>2\cdot\left(2^{n}-n-1\right)$ then we have $\delta_{\alpha}\left(B/A\right)<\frac{1}{2}\cdot\delta_{\alpha}\left(A\right)$
and therefore embeddings of $B$ and $A$ form a free 2-pseudoplane.
\end{proof}
Now using Theorem $\ref{thm:2-pseudo}$ we obtain the following.
\begin{cor}
\label{cor:-is-not} $\Aut\left(\mathcal{\mathbf{M}^{\alpha}}\right)$
is not amenable.
\end{cor}

\section{Further remarks and questions}
In this paper we have shown the automorphism group of certain generic structures are not amenable.  The following question is a natural question to ask

\begin{question}
	Is there an example of a smooth class $\left(\mathcal{C},\leqslant\right)$
	with HP and JEP such the notion of closedness notion of $\leqslant$ is strictly
	stronger that $\subseteq$, and the automorphism group of the $\left(\mathcal{C},\leqslant\right)$-generic
	structure is amenable?
\end{question}

 Note that in the question above we asked $\leqslant$ to be strictly
 stronger that $\subseteq$ to insure the generic structure is not the usual  Fra\"iss\'e-limit structure and the reason is that in  Fra\"iss\'e-limit structures there is no pair that form a free-pseudoplane. It is interesting to find a border line that also grasps the geometric behavior of the algebraic closure in the generic and connects it with the amenability of the automorphism group.   
 
 It is also interesting to determine whether or not in the automorphism group of the generic structures that we proved the non amenability for one can find a copy of $\mathrm F_2$ that acts freely on structure  or a substructure of the generic model (see \cite{MR3417738}). Another property to verify is Kazhdan Property (T) and its connection to amenability in these cases, although it is known there is no direct connection between amenability and  Kazhdan Property (T) (see \cite{MR2929072} for more details).

\bibliographystyle{amsplain}
\providecommand{\bysame}{\leavevmode\hbox to3em{\hrulefill}\thinspace}
\providecommand{\MR}{\relax\ifhmode\unskip\space\fi MR }
\providecommand{\MRhref}[2]{%
  \href{http://www.ams.org/mathscinet-getitem?mr=#1}{#2}
}
\providecommand{\href}[2]{#2}


\end{document}